\documentclass{amsart}

\usepackage[utf8]{inputenc}
\usepackage{xcolor}
\usepackage{todonotes}
\usepackage{amssymb}
\usepackage{nicefrac}
\usepackage{enumerate}
\usepackage{natbib}

\usepackage{hyperref}
\hypersetup{
    colorlinks=true,
    linkcolor=teal,
    citecolor=magenta,
    }
\usepackage[nameinlink, capitalize, noabbrev]{cleveref}

\newtheorem{Theorem}{Theorem}

\newtheorem{proposition}{Proposition}[section]
\newtheorem{lemma}[proposition]{Lemma}
\newtheorem{corollary}[proposition]{Corollary}
\newtheorem{theorem}[proposition]{Theorem}

\theoremstyle{definition}
\newtheorem{example}[proposition]{Example}
\newtheorem{remark}[proposition]{Remark}

\DeclareMathOperator\Aut{Aut}

\newcommand*\acts{\curvearrowright}
\newcommand*\abs[1]{\lvert#1\rvert}

\newcommand*\gen[1]{\langle#1\rangle}

\newcommand*\setst[2]{\{#1\,|\,#2\}}

\numberwithin{equation}{section}

\title{Non-closed subgroups of weakly branch groups}
\author{Jorge Fariña-Asategui, Paul-Henry Leemann, Tatiana Nagnibeda}
\address{Jorge Fariña-Asategui: Centre for Mathematical Sciences, Lund University, 223 62 Lund, Sweden -- Department of Mathematics, University of the Basque Country UPV/EHU, 48080 Bilbao, Spain}
\email{jorge.farina\_asategui@math.lu.se}
\address{Paul-Henry Leemann: Department of Pure Mathematics, Xi'an Jiaotong-Liverpool University, 215123 Suzhou, China}
\email{paulhenry.leemann@xjtlu.edu.cn}
\address{Tatiana Nagnibeda: Section de math\'ematiques, Universit\'e de Gen\`eve, 7-9 rue G\'en\'eral-Dufour, 1204 Gen\`eve, Switzerland}
\email{tatiana.smirnova-nagnibeda@unige.ch}
\thanks{The first author is supported by the Spanish Government, grant PID2020-117281GB-I00, partly with FEDER funds and by the Walter Gyllenberg Foundation from the Royal Physiographic Society of Lund. The authors gratefully acknowledge support from the Swiss NSF grant 200020-200400. The second author is supported by RDF-23-01-045-Groups acting on rooted trees and their subgroups of XJTLU}

\begin{document}

\begin{abstract}

For a weakly branch group $G$ acting on a regular enough rooted tree, we provide two constructions of continuous families of distinct subgroups that are not closed in the profinite topology on $G$. On the one hand, we construct a continuous family of distinct non-closed subgroups such that each $H$ in the family is not ERF, that is, contains subgroups not closed in the profinite topology on $H$. On the other hand, under an additional assumption on $G$, we construct a continuous family of ERF subgroups which are not closed in the congruence (and in the profinite) topology on $G$.  

\end{abstract}
\maketitle

\section{introduction}
\label{section: introduction}

Given a group, one can endow it with the so-called \emph{profinite topology},  where finite-index subgroups form a basis of open neighbourhoods of the identity. This is a natural topology to consider on a group. For example, groups for which the trivial subgroup is closed in the profinite topology are exactly the  \emph{residually finite groups}, i.e., groups that can be approximated by their finite quotients, in the sense that every non-trivial group element of $G$ remains non-trivial in some finite quotient of $G$. A more restrictive natural condition requires that all finitely generated subgroups be closed in the profinite topology; such groups are called \emph{subgroup separable} or \emph{LERF} (short for \emph{locally extended residually finite}). By a celebrated theorem of Marshall Hall, free groups are LERF \cite{Hall}. Further examples of LERF groups notably include all limit groups \cite{Wilton} and the first Grigorchuk group \cite{GW}.

An even stronger notion is the so-called \emph{ERF} (short for \emph{extended residually finite}) property, which is satisfied for a group $G$ if all its subgroups  are closed in the profinite topology. It was introduced in 1958 by Mal'cev, who proved that all  virtually polycyclic groups are ERF \cite{Mal}. The ERF property has been completely characterized in certain classes of groups, such as abelian or nilpotent groups; see \cite{RRV} and the references therein. Virtually polycyclic groups remain the only known ERF examples among finitely generated groups.

In this paper, we will restrict our attention  to residually finite groups, or equivalently, to subgroups of automorphisms of   spherically homogeneous rooted trees. Of particular interest to us are branch groups, a class of residually finite groups with rich subnormal subgroup structure similar to that of the full group of automorphisms of the tree \cite{GrigNH,Wilson}. The class of branch groups contains examples of groups of intermediate growth, of amenable but not elementary amenable groups, of infinite finitely generated torsion groups, of self-similar groups. Further interesting examples belong to the class of weakly branch groups.

Grigorchuk and Wilson proved in \cite{GW} that the first Grigorchuk group is LERF. This result was extended to the Gupta-Sidki $3$-group by Garrido \cite{Gar16} and further generalized by Francoeur and  the second author to branch groups satisfying the so-called subgroup induction property  \cite{FL}. This is a strong property which implies in particular that the group is torsion. These groups also satisfy a remarkable property \emph{(MF)} that all of its maximal subgroups are of finite index, first shown by Pervova \cite{Pervova1, Pervova2}. All their weakly maximal\footnote{A subgroup is weakly maximal is it is maximal among subgroups of infinite index.} subgroups are closed in the profinite topology too.

Observe that a group that is ERF is necessarily \emph{(MF)}, as a maximal subgroup of infinite index must be dense in the profinite topology. The property \emph{(MF)} has been shown to hold for many (weakly) branch groups, and one can wonder whether examples of ERF groups can be found among them. Cornulier pointed out in \cite[Example~3.10(8)]{Cornulier} that the first Grigorchuk group, a famous example of a branch group acting on the binary tree, contains a non-ERF subgroup and thus, it is itself non-ERF. 
In this note we  show that weakly branch groups are non-ERF by providing two constructions of countable families of pairwise distinct subgroups which are not closed in profinite topolgy. The first one is Cornulier's example extended to all weakly branch groups (and in fact more generally, to groups with micro-supported actions). Our second result is a sufficient condition for a group acting by automorphisms on a spherically homogeneous rooted tree to contain a continuum of subgroups that are not close in the congruence topology, despite being ERF.  Before formulating our results, let us introduce some notation and terminology.

Consider a sequence $\mathbf{m}:=\{m_n\}_{n\ge 1}$, where $m_n\ge 2$ for every $n\ge 1$. The corresponding spherically homogeneous rooted tree $T_\mathbf{m}$ has  $m_n$  immediate descendants in each vertex at level $n-1$. If $m_n=m$ for all $n\geq 1$, we write $T_m$ for the corresponding regular rooted tree. 
For  a vertex $v\in T_\mathbf{m}$, we denote by $T_\mathbf{m}^v$  the subtree rooted at $v$.
Given a group $G \le \mathrm{Aut}~T_\mathbf{m}$, we write $\mathrm{st}_G(v)$ for  the stabilizer of a vertex $v\in T_\mathbf{m}$ in~$G$ and $\mathrm{st}_G(n)$ for the (pointwise) stabilizer of the $n$-th level of $T_\mathbf{m}$ in $G$.
We also consider $\mathrm{rist}_G(v)$, the \emph{rigid stabilizer} of $v$, that is, the subgroup of $G$ consisting of the elements in $G$ which act trivially on $T_\mathbf{m}\setminus T_\mathbf{m}^v$. Note that  $[\mathrm{rist}_G(u),\mathrm{rist}_G(v)]=1$ if $u$ and $v$ are \textit{incomparable}, that is, $u\ne v$ and none is a descendant of the other. For $n\ge 1$, the direct product $\mathrm{rist}_G(n):=\prod \mathrm{rist}_G(v)\le G$, where the product ranges over all the vertices at the $n$th level of $T_\mathbf{m}$, is a normal subgroup of $G$ called the  \emph{$n$th rigid level stabilizer} of $G$. 
 
 A group $G\le \mathrm{Aut}~T_\mathbf{m}$ that acts transitively on all levels of the tree $T_\mathbf{m}$ is \emph{weakly branch} if $\mathrm{rist}_G(n)$ is non-trivial for every $n\ge 1$, and \emph{branch} if $\mathrm{rist}_G(n)$ is of finite index in~$G$ for every $n\ge 1$.

\begin{Theorem}
    \label{Theorem: weakly branch are not erf}
    Let $T_\mathbf{m}$ be a spherically homogeneous rooted tree such that\footnote{In particular, \cref{Theorem: weakly branch are not erf} applies to all weakly branch groups that act on  regular rooted trees, which are the most-studied weakly branch groups in the literature.}   there are only finitely many primes dividing the terms in the sequence $\mathbf{m}$. Let $G\le\mathrm{Aut}~T_\mathbf{m}$ be a weakly branch group. Then $G$ contains a continuum of pairwise distinct non-closed  subgroups that are not ERF. In particular $G$ is not ERF.
\end{Theorem}

In fact, the construction in \cref{Theorem: weakly branch are not erf} works in the more general context of groups with micro-supported actions, see \cref{section: non-erf subgroups}.

For $G\le \mathrm{Aut}~T_\mathbf{m}$, the \emph{congruence topology} is defined by a basis of open neighbourhoods of the identity formed by the level stabilizers  $\{\mathrm{st}_{G}(n)\}_{n\ge 1}$. 
Every such stabilizer is of finite index in $G$, hence the congruence topology is coarser than the profinite topology. 
In some cases, they coincide, and  then $G$ is said to satisfy the \emph{congruence subgroup property}.

\begin{Theorem}
\label{Theorem: uncountably many subgroups}
    Let $G\le \mathrm{Aut}~T_{\mathbf{m}}$.  Suppose that  there exists a prime $p\ge 2$ such that for every $v\in T_\mathbf{m}$, the subgroup $\mathrm{rist}_G(v)$ has an element of order divisible by~$p$.
    Then~$G$ contains a continuum of pairwise distinct ERF subgroups which are non-closed in the congruence topology.
\end{Theorem}

Note that the assumption in \cref{Theorem: weakly branch are not erf} of only finitely many primes dividing the terms in the sequence $\mathbf{m}$, together with every rigid vertex stabilizer containing torsion elements, imply the assumption of \cref{Theorem: uncountably many subgroups}.

The assumption in \cref{Theorem: uncountably many subgroups} is satisfied by many well-known (weakly) branch groups in the literature; we mention some of them in the end of the paper. The proof of \cref{Theorem: uncountably many subgroups} is constructive and yields explicit examples of subgroups which are not closed in both the congruence topology and the profinite topology.

The paper is organized as follows. In \cref{section: non-erf subgroups}, we recall what is known on the ERF property for abelian groups and then prove \cref{Theorem: weakly branch are not erf}. In \cref{section: main proofs}, we prove \cref{Theorem: uncountably many subgroups} and show that it applies to some well-known examples of groups of automorphisms of regular rooted trees.

\subsection*{Acknowledgements} The first and second authors would like to thank the Université de Genève for its warm hospitality while this work was carried out.

\section{Non-ERF subgroups of groups with micro-supported actions}
\label{section: non-erf subgroups}

The goal of this section is to establish \cref{Theorem: weakly branch are not erf}. In fact, we will prove this result in a more general context  of groups with micro-supported actions; see \cref{theorem: main result for micro supported actions} below.

\subsection{Groups with micro-supported actions}

 Let us consider a group $G$ acting by homeomorphisms on a topological space $X$. For any non-empty open subset $U\subseteq X$, the pointwise fixator of the complement $X\setminus U$ is called the \emph{rigid stabilizer} of $U$, and we denote it by $\mathrm{rist}_G(U)$. We say that the action of $G$ on $X$ by homeomorphisms is a \emph{micro-supported action} if $\mathrm{rist}_G(U)\ne 1$ for every non-empty open subset $U\subseteq X$.

If $G$ acts on a spherically homogeneous rooted tree $T_\mathbf{m}$, then its action on $T_\mathbf{m}$ induces an action on the boundary $\partial T_\mathbf{m}$. The induced action on the boundary $\partial T_\mathbf{m}$ is minimal if and only if the action on $T_\mathbf{m}$ is level-transitive, while the action on $\partial T_\mathbf{m}$ is micro-supported if and only if the $\mathrm{rist}_G(n)$ are infinite. Hence $G$ is weakly branch if and only if its action on $\partial T_\mathbf{m}$ is both minimal and micro-supported, and weakly branch groups are therefore examples of groups with micro-supported actions. In general, a micro-supported action doesn't have to come from an action on a spherically homogeneous rooted tree. In particular, there are groups with micro-supported actions that are not residually finite (and hence obviously non-ERF), for example, Thompson groups.

\subsection{ERF property for Abelian groups}

Abelian groups and, more concretely, direct sums of cyclic groups, play a central role in the proof of \cref{Theorem: weakly branch are not erf}. Hence, we begin by reviewing the ERF property for this class of groups. 

Not all abelian groups are ERF. In fact, ERF groups are residually finite, so non-residually finite abelian groups like the Prüfer groups $\mathbb{Z}_{p^\infty}$ cannot be ERF. Since the ERF property is preserved by taking factors of subnormal subgroups, an ERF abelian group cannot have a subnormal subgroup with a factor isomorphic to $\mathbb{Z}_{p^\infty}$ for any prime $p$.
Following the terminology in \cite{RRV}, we call such groups \emph{Prüfer-free}. In general, we have the following characterisation of abelian ERF groups due to Robinson, Russo and Vincenzi (note that for an abelian group $A$, we write $\tau(A)$ for the torsion subgroup of $A$ and $A_p$ for the $p$-component of $A$ for a prime $p$):

\begin{theorem}[{see {\cite[Proposition 3.1]{RRV}}}]
\label{Thm:RRV}
Let $A$ be an abelian group. Then $A$ is ERF if and only if the following two conditions hold:
\begin{enumerate}[\normalfont(i)]
    \item $A/\tau(A)$ is torsion-free, of finite Prüfer rank and Prüfer-free;
    \item $A_p$ has finite exponent for all primes $p$.
\end{enumerate}
\end{theorem}

Recall that the Prüfer rank of a torsion-free abelian group $A$ is the dimension of the smallest $\mathbb Q$-vector space containing $A$.
In particular, if $A$ is a direct sum of copies of $\mathbb Z$, then the Prüfer rank of $A$ coincide with the rank of $A$ (the infimum of the size of generating sets).

As mentioned above, we are particularly interested in direct sums of cyclic groups. In this context, we obtain the following:

\begin{corollary}\label{Cor:SumERF}
Let $A:=\bigoplus_{v\in V}A_v$ where $A_v$ is cyclic for each $v\in V$. Write 
$$I:=\setst{v\in V}{A_v\cong\mathbb{Z}}\quad\text{and}\quad F:=V\setminus I$$
for the subset of indices in $V$ corresponding to infinite cyclic subgroups and for the subset of indices in $V$ corresponding to finite cyclic subgroups respectively.
Furthermore, let 
$$\mathcal P:=\setst{p \textnormal{ prime}}{\exists v\in F \text{ such that } p\textnormal{ divides }\abs{A_v}}$$
be the set of prime divisors of the orders of the finite cyclic groups $\{A_v\}_{v\in F}$. Then, the following holds:
\begin{enumerate}[\normalfont(i)]
\item if $I$ is infinite, then $A$ is not ERF;
\item if $I$ is finite and the groups in $\{A_v\}_{v\in F}$ have uniformly bounded orders, then $A$ is ERF;
\item if $\mathcal P$ is finite and the groups in $\{A_v\}_{v\in F}$ have unbounded orders, then $A$ is not ERF.
\end{enumerate}
\end{corollary}
\begin{proof}
The group $A/\tau(A)\cong \bigoplus_{\abs I}\mathbb{Z}$  is always torsion-free and has rank $\abs I$.
It follows that $A/\tau(A)$ has finite rank if and only if $I$ is finite, if and only if $A/\tau(A)$ is Prüfer-free.
In other words, the first condition of \cref{Thm:RRV} is satisfied if and only if $I$ is finite.

The second condition of \cref{Thm:RRV} is satisfied as soon as the orders $\{|A_v|\}_{v\in F}$ are uniformly bounded on $F$. If the orders $\{|A_v|\}_{v\in F}$ are unbounded and $\mathcal P$ is finite, then there exists some $p\in \mathcal{P}$ such that $A_p$ has unbounded exponent and thus, the second condition of \cref{Thm:RRV} fails.
\end{proof}

If the orders $\{|A_v|\}_{v\in F}$ are unbounded and $\mathcal P$ is infinite, then it is both possible that $A_p$ has finite exponent for every prime $p\in \mathcal{P}$, or that there is $p\in\mathcal{P}$ such that~$A_p$ has infinite exponent. In this case, without any further information, one cannot conclude whether $A$ is ERF or not; see \cref{Ex:Bounded} and \cref{Ex:BoundedNonTorsion}.

Observe that if $T_{\mathbf{m}}$ is \emph{bounded} (i.e. there exists $M$ such that $m_n\le M$ for all $n\ge 1$) and $A=\bigoplus_{v\in V}A_v$ is a subgroup of $\Aut~T_{\mathbf{m}}$ which decomposes into a sum of cyclic groups, then $\mathcal P$ is finite.
Indeed, for any $g\in\Aut~T_{\mathbf{m}}$ of finite order $o(g)$ and any $p$ prime divisor of $o(g)$, we necessarily have $p\le \max_{n\ge 1}\{m_n\}\le M$. In particular, this holds for any regular rooted tree~$T_m$.

\subsection{Non-ERF subgroups in groups with micro-supported actions}

We now generalize the first example of a non-ERF subgroup of the first Grigorchuk group given by Cornulier in \cite[Example 3.10(8)]{Cornulier}.

We need the following lemma, whose proof is sketched, for weakly branch groups, in~\cite[Lemma 6.8]{BGZ}. We provide a complete proof in the more general context of micro-supported actions:

\begin{lemma}\label{Lem:CyclicSub}
Let $X$ be a Hausdorff space. Let $G\acts X$ be a micro-supported action and let $U\subseteq X$ be a non-empty open subset of $X$. Then for any $N\ge 2$, the subgroup $\mathrm{rist}_G(U)$ contains an element of order at least $N$. Moreover, if $\mathrm{rist}_G(U)$ contains an element of finite order for every non-empty open $U\subseteq X$, then $\mathrm{rist}_G(U)$ contains elements of finite order at least $N$.
\end{lemma}
\begin{proof}
First note that it is enough to produce elements in $\mathrm{rist}_G(U)$ with arbitrarily large orbits in~$X$, as for any $g\in G$ and any $x\in X$ the length of the $g$-orbit of $x$ divides the order of $g$.

Let $U_1:=U$ and consider $g_1$ a non-trivial element of $\mathrm{rist}_G(U_1)$. Let $x_1\in X$ be a point moved by $g_1$, whose $g_1$-orbit is of length $k_1\ge 2$. Since $X$ is Hausdorff, one can find an open neighbourhood $U_2\subseteq U_1$ of $x_1$ such that $U_2,U_2^{g_1},\dots,U_2^{g_1^{k_1-1}}$ are pairwise disjoint. Now, let $g_2$ be any non-trivial element of $\mathrm{rist}_G(U_2)$ and $x_2\in U_2$ any point moved by~$g_2$. We claim that the $g_1g_2$-orbit of $x_2$ is of length $k_1k_2$ for some $k_2\ge 2$. Indeed, since~$g_2$ is in $\mathrm{rist}_G(U_2)$ and $U_2,U_2^{g_1},\dots,U_2^{g_1^{k_1-1}}$ are pairwise disjoint, $g_2$ acts trivially on $U_2^{g_1},\dots,U_2^{g_1^{k_1-1}}$ and we get $$x_2^{(g_1g_2)^i}=x_2^{g_1^i}\ne x_2^{g_1^j}=x_2^{(g_1g_2)^j}$$
for $0\le i < j\le k_1-1$. Moreover, as
$$(g_1g_2)^{k_1}=g_1^{k_1}g_2^{g_1^{k_1-1}}\dotsb g_2^{g_1}g_2,$$
we further get 
$$x_2^{(g_1g_2)^{k_1}}=
x_2^{g_2}\ne x_2.$$
Therefore, as the $g_1g_2$-orbit of $x_1$ is of length $k_1$, the length of the $g_1g_2$-orbit of $x_2$ is $k_1k_2$ for some $k_2\ge 2$.

Repeating the above procedure, one can find a sequence of non-empty open subsets
$$U=:U_1\supsetneq U_2\supsetneq\dotsb \supsetneq U_i\supsetneq\dotsb,$$
and for each $i\ge 1$ an element $g_i \in\mathrm{rist}_G(U_{i})$ such that $g_1g_2\cdots g_n\in \mathrm{rist}_G(U)$ has an orbit of length $k_1k_2\dotsb k_n\ge 2^n$ for every $n\ge 1$. This yields elements in $\mathrm{rist}_G(U)$ with arbitrarily large orbits in $X$, concluding the proof of the first assertion.

For the second assertion, if for every non-trivial open subset $U\subseteq X$ the subgroup $\mathrm{rist}_G(U)$ contains elements of finite order, then the elements $g_1g_2\dotsb g_n$ can be further assumed to be of finite order. Indeed, by possibly taking powers, we may assume that each $g_i$ has prime order $p_i$. Then $x_1$ can be chosen such that the $g_1$-orbit of $x_1$ is of length $p_1$. Therefore $(g_1g_2)^{p_1}$ maps each $U_2^{g_1^i}$ to itself and acts trivially on $X\setminus (\bigsqcup_{i} U_2^{g_1^i})$, where $0\le i\le p_1-1$. Furthermore the restriction of $(g_1g_2)^{p_1}$ to each $U_2^{g_1^i}$ coincides with the restriction of $g_2^{g_1^i}$ to~$U_2^{g_1^i}$, so $(g_1g_2)^{p_1}$ has order $p_2$. Hence $g_1g_2$ has order $p_1p_2$. Iterating this argument we get that $g_1g_2\dotsb g_n$ has order $p_1p_2\dotsb p_n$.
\end{proof}

\begin{theorem}
    \label{theorem: main result for micro supported actions}
    Let $X$ be a Hausdorff space and let $G$ be a group such that:
    \begin{enumerate}[\normalfont(i)]
        \item $G$ acts on $X$ via a micro-supported action $G\acts X$;
        \item there are only finitely many primes dividing the order of the torsion elements in $G$.
    \end{enumerate}
    Then $G$ contains a continuum of pairwise distinct 
    non-closed and
    non-ERF subgroups so, in particular, $G$ is not ERF. Furthermore, if every rigid stabilizer of $G$ contains a torsion element (for example if $G$ is torsion), then $G$ has a continuum of pairwise non-isomorphic
    non-closed and
    non-ERF subgroups.
\end{theorem}
\begin{proof}

Let $\mathcal{U}:=\{U_i\}_{i\ge 1}$ be an infinite collection of pairwise disjoint non-empty subsets of $X$.
Then, by \cref{Lem:CyclicSub}, there exists a subcollection $\mathcal V:=\{V_i=U_{r_i}\}_{i\geq 1}\subseteq \mathcal U$ and an abelian subgroup $H_\mathcal{V}\le G$ such that:
\begin{enumerate}[\normalfont(i)]
    \item $H_\mathcal{V}:=\bigoplus_{i\ge 1}H_i$;
    \item each $H_i\le \mathrm{rist}_G(V_i)$ is cyclic;
    \item either $|H_i|=\infty$ for all $i$, or there exists a prime $p$ such that $|H_i|=p^{n_i}$ for some strictly increasing sequence $\{n_i\}_{i\geq 1}$.\label{itemthree}
\end{enumerate}
Thus $H_\mathcal{V}$ is non-ERF by \cref{Cor:SumERF}.
We will first prove the Theorem under the assumption that the $H_i$ are all finite. This will be the case if every rigid stabilizer of $G$ contains torsion elements, by \cref{Lem:CyclicSub}.
Let $p$ be the prime given by \textcolor{teal}{(}\ref{itemthree}\textcolor{teal}{)} above.
Then $H_\mathcal{V}$ projects onto the Prüfer group $\mathbb{Z}_{p^\infty}$. Let~$K_\mathcal{V}$ be the kernel of this projection. For $j\geq i\ge 1$, let $f_j^i\colon \mathbb{Z}/p^{n_i} \mathbb{Z} \to \mathbb{Z}/p^{n_j} \mathbb{Z}$ be the standard inclusion given by $x\mapsto p^{n_j-n_i}x$. 
Then, an element $(x_1,x_2,\dots, x_m,0,\dots)\in H_\mathcal{V}$ is in $K_\mathcal{V}$ if and only if $\sum_{i=1}^mf_m^i(x_i)\equiv 0 \pmod{p^{n_m}}$.

The subgroup $K_\mathcal{V}$ is a torsion abelian $p$-group of infinite rank, which is not closed in $G$. Indeed, as the subgroup ~$H_\mathcal{V}$ is closed in $G$ (in the congruence topology and thus also in the profinite topology), not being closed in $G$ follows from $K_\mathcal{V}$ not being closed in $H_\mathcal{V}$. The latter follows from $K_\mathcal{V}$ being the kernel of the projection of $H_\mathcal{V}$ onto the Prüfer group $\mathbb{Z}_{p^\infty}$ which is not residually finite. Moreover, it follows from the description of $K_\mathcal{V}$ that it contains elements of arbitrary large order and is hence not ERF by \cref{Thm:RRV}.

Let us show now that this construction produces indeed a continuum of pairwise disjoint subgroups. For any binary sequence $\mathbf b:=(b_i)_{i\geq 1}$ with infinitely many $1$s, define $\mathcal V_{\mathbf b}$ to be the subset of $\mathcal V$ consisting of the $V_i$ such that $b_i=1$. Then, if $\mathbf b\neq  \widetilde{\mathbf b}$, then there exists $i\ge 1$, such that exactly one of $K_{\mathcal{V}_{\mathbf b}}$ and $K_{\mathcal{V}_{\widetilde{\mathbf b}}}$ acts non-trivially on~$V_i$.
In particular, the subgroups $K_{\mathcal{V}_{\mathbf b}}$ and $K_{\mathcal{V}_{\widetilde{\mathbf b}}}$ cannot be equal.

Now, let $k_i:=(0,\dots, 0,1,-p^{n_{i+1}-n_i},0\dots)$ be the element of $H_\mathcal{V}$ with a $1$ in position $i$, a $-p^{n_{i+1}-n_i}$ in position $i+1$ and $0$s everywhere else. Then the $k_i$'s generate $K_\mathcal{V}$ and $K_\mathcal{V}$ is the direct sum of the cyclic subgroups $\gen{k_i}$, where each $k_i$ has order $p^{n_i}$.
Therefore, $K_{\mathcal{V}_{\mathbf b}}$ is abstractly isomorphic to $\bigoplus_{i\geq 1}\mathbb{Z}/p^{b_in_i}\mathbb{Z}$.
However, the decomposition of a $p$-group into direct sums of cyclic groups of prime power orders is unique up to permutation of the factors, see \cite[Theorem 17.4.]{Fuchs}, so $K_{\mathcal{V}_{\mathbf b}}$ and $K_{\mathcal{V}_{\widetilde{\mathbf b}}}$ are isomorphic if and only if $\mathbf b=\widetilde{\mathbf b}$.

We finally treat the case where all the $H_i$ are infinite.
In this case, $H_{\mathcal{V}}$ is isomorphic to $\bigoplus_{i\geq 1}\mathbb{Z}$, which projects onto $\bigoplus_{i\geq 1}\mathbb{Z}/2^i\mathbb{Z}$ and hence onto $\mathbb{Z}_{2^\infty}$.
Let $L_{\mathcal{V}}$ be the kernel of the projection $H_{\mathcal{V}}\to \mathbb{Z}_{2^\infty}$.
Then similarly to the above case, we have that $L_{\mathcal{V}}$ is a non-closed subgroup of $G$, it has infinite rank and it is isomorphic to $\bigoplus_{i\geq 1}\mathbb{Z}$ and thus not ERF.
We can once again construct a continuum of pairwise distinct such subgroups $L_{\mathcal{V}_{\mathbf b}}$, indexed by binary sequences containing infinitely many $1$s. However, this time, all the $L_{\mathcal{V}_{\mathbf b}}$ are abstractly isomorphic.
\end{proof}

If $G\leq \Aut~T_{\mathbf{m}}$, then a natural choice for $\mathcal U$ in the proof of \cref{theorem: main result for micro supported actions} is $\{\partial T_{\mathbf{m}}^{v_i}\}_{i\ge 1}$ where the $v_i$ are pairwise incomparable vertices of $T_{\mathbf{m}}$.

\subsection{Examples}

If $G$ is torsion, the condition of only finitely many primes dividing the order of the torsion elements of $G$ cannot be dropped from \cref{theorem: main result for micro supported actions}.

Let $G$ be a torsion branch group such that infinitely many primes divide the order of the elements of $G$.
This property passes to finite index subgroups, and in particular to rigid stabilizers of vertices.
Therefore, for any collection of pairwise incomparable vertices $V=\{v_i\}_{i\geq 1}$, there exists ERF subgroups of the form $H_V=\bigoplus_{i\geq 1}H_i$ with each $H_i\leq\mathrm{rist}_G(v_i)$ cyclic.
Indeed, it is enough to take the subgroups~$H_i$ of increasing prime order.

Observe that, in this generality, we can only conclude that some of the $H_V$ we construct are ERF, but we cannot guarantee that all of them are ERF. Indeed, rigid stabilizers might contain elements of order a power of a fixed prime $p$. In any case, the following example shows that it may happen that every subgroup of the form $H_V$ is indeed ERF:

\begin{example}
\label{Ex:Bounded}
Let $(p_i)_{i\ge 1}$ be a sequence of prime numbers with $\liminf p_i=\infty$, and let $T:=T_{(p_i)_{i\ge 1}}$.
For each $i\ge 1$, let $v_i$ be a vertex at the level $i$ in $T$ and let $a_i$ be the rooted automorphism of $T^{v_i}$ cyclically permuting the $p_i$ children of~$v_i$.
Then the group $G:=\langle a_i\mid i\ge 1\rangle$ is a torsion branch group. 
Moreover, for any prime $p$, there exists a level $i_p$ in $T$ such that for all $i\geq i_p$ we have $p_i>p$.
By the construction of $G$, this implies that any element of $G$ has $p$-exponent at most~$p^{i_p}$.
In particular, for any collection of pairwise incomparable vertices $V=\{v_i\}_{i\geq 1}$, any subgroup of the form $H_V=\bigoplus_{i\geq 1}H_i$ with $H_i\leq\mathrm{rist}_G(v_i)$ cyclic is necessarily ERF by \cref{Thm:RRV}.
\end{example}

If $G$ is a branch group containing an element of infinite order, then the construction in \cref{theorem: main result for micro supported actions} always yields a non-ERF subgroup of $G$:

\begin{remark}
\label{Ex:BoundedNonTorsion}

For a branch group $G$, it is equivalent that $G$ contains an element of infinite order and that every $\mathrm{rist}_G(v)$ contains an element of infinite order. Indeed, if $g\in G$ is of infinite order, for any $n\ge 1$, some power of $g$ is in $\mathrm{rist}_G(n)$ and it has infinite order. However, as $\mathrm{rist}_G(n)$ is the direct product of the rigid vertex stabilizers of vertices at the $n$th level, it contains an element of infinite order if and only if one of the rigid vertex stabilizers (and all of them by level-transitivity) contains an element of infinite order. Thus, if $G$ is branch, every $\mathrm{rist}_G(v)$ contains an element of infinite order if and only if $G$ is not torsion.
\end{remark}

\section{Non-closed ERF subgroups of weakly branch groups}
\label{section: main proofs}

In this section we prove \cref{Theorem: uncountably many subgroups} and illustrate it with some examples.

\subsection{The main proof}

Before proving \cref{Theorem: uncountably many subgroups}, we set some notation. Let $u\in T_\mathbf{m}$ and $v\in T_\mathbf{m}^u$. Then, we write $uv$ for the unique vertex in $T_\mathbf{m}$ corresponding to~$v$.

\begin{proof}[Proof of \cref{Theorem: uncountably many subgroups}]
    Fix $p\ge 2$ a prime satisfying condition (i). Let $a\in G$ be an element of order $p$ and $h_v\in \mathrm{rist}_G(v)$ an element of order $p$ for each $v\in T_\mathbf{m}$. Since $a$ has order $p$, let us consider $n_0$ a level of $T_\mathbf{m}$ at which there is $a$-orbit $\{x, x^a \dotsc, x^{a^{p-1}}\}$ of size $p$.
    
    Let us consider an infinite set $U_x$ of pairwise incomparable vertices in $T_\mathbf{m}^{x}$ and let us further define $V_{x}:=\{xv\mid v\in U_{x}\}$. Then $V_x^{a^i}\subset T_\mathbf{m}^{x^{a^i}}$ for $0\le i\le p-1$, and we may define the disjoint union
     $$\mathcal{V}:=V_{x}\sqcup V_{x}^a \sqcup\dotsb \sqcup V_{x}^{a^{p-1}}.$$
     Note that $\mathcal{V}$ is a set of pairwise incomparable vertices in $T_\mathbf{m}$ and that we have the equality $\mathcal{V}^a=\mathcal{V}$, as $a$ has order $p$ and the $a$-orbit of $x$ is of size $p$. 
     
     Now $ah_v$ is of order $p^2$ for every $v\in \mathcal{V}$. Indeed 
     \begin{align*}
         (ah_v)^{p^2}&=((ah_v)^p)^p=(a^ph_v^{a^{p-1}}\dotsb {h_v^{a}}h_v)^p=(h_v^{a^{p-1}}\dotsb {h_v^{a}}h_v)^p\\
         &=(h_v^{a^{p-1}})^p\dotsb ({h_v^{a}})^p(h_v)^p=1,
     \end{align*}
     as $v,v^a,\dotsc,v^{a^{p-1}}\in \mathcal{V}$ are incomparable vertices, so $h_v,h_{v}^a,\dotsc,h_{v}^{a^{p-1}}$ are commuting elements of order $p$.
     
     Let us define the subgroup 
     $$H_\mathcal{V}:= \langle ah_v\mid v\in V_x\rangle.$$

     Then $a\in \overline{H}_\mathcal{V}$. Indeed, as $V_x$ is infinite, for each $n\ge 1$ there exists $v\in V_x$ below level $n$, so 
     \begin{align}
     \label{align: first profinite}
         a=ah_v\cdot h_v^{-1}\in H_\mathcal{V}\mathrm{st}_G(n).
     \end{align}
     Thus 
     \begin{align}
     \label{align: second profinite}
         a\in \bigcap_{n\ge 1}H_\mathcal{V}\mathrm{st}_G(n)=\overline{H}_\mathcal{V}.
     \end{align}

     We shall prove $a\notin H_\mathcal{V}$, which yields $\overline{H}_\mathcal{V}>H_\mathcal{V}$ proving $H_\mathcal{V}$ is not closed in~$G$ in the congruence topology. Assume by contradiction that $a\in H_\mathcal{V}$. Since the generators of $H_\mathcal{V}$ are of finite order, we may assume that $a$ can be written as a finite word
     \begin{align*}
         a=ah_{v_1}ah_{v_2}\dotsb ah_{v_{r-1}}ah_{v_r},
     \end{align*}
     where $v_1,\dotsc,v_r\in V_x$. Since $h_{v_1},\dotsc, h_{v_r}\in \mathrm{st}_G(n_0)$ by assumption, we must have $a=a^r$ modulo $\mathrm{st}_G(n_0)$. Thus $a^r=a$ as $\langle a\rangle$ acts faithfully on level $n_0$. In particular $r\equiv 1$ mod $p$ as $a$ has order~$p$. Therefore
     \begin{align}
     \label{align: erf equation}
         a&=ah_{v_1}ah_{v_2}\dotsb ah_{v_{r-1}}ah_{v_r} =a^rh_{v_1}^{a^{r-1}}\dotsb h_{v_{r-1}}^{a} h_{v_r},
     \end{align}
     and simplifying $a$ from both sides we obtain the relation
     \begin{align}
     \label{align: product of hv}
         h_{v_1}^{a^{r-1}}\dotsb h_{v_{r-1}}^{a} h_{v_r}&=1.
     \end{align}
     However, since $\{v_1^{a^{r-1}},\dotsc, v_{r-1}^{a},v_r\}\subset \mathcal{V}$, all these vertices are either equal or pairwise incomparable, so the non-trivial elements $h_{v_1}^{a^{r-1}},\dotsc, h_{v_{r-1}}^{a}, h_{v_r}$ commute. Hence
     $$h_{v_1}^{a^{r-1}}\dotsb h_{v_{r-1}}^{a} h_{v_r}\ne 1$$
     as $r\equiv 1$ mod $p$ and $h_v$ has order $p$ for every $v\in T_\mathbf{m}$. This contradicts \cref{align: product of hv}. Thus $a\notin H_\mathcal{V}$. 

     To show that $H_\mathcal{V}$ is ERF just note that the finite index subgroup $\mathrm{st}_{H_\mathcal{V}}(1)$ is ERF by \cref{Cor:SumERF}, since we have
     $$\mathrm{st}_{H_\mathcal{V}}(1)\le \langle h_v^{a^i}\mid v\in V_x\text{ and }0\le i\le p-1\rangle\cong \bigoplus_{v\in \mathcal{V}}\mathbb{Z}/p\mathbb{Z}$$
     arguing as in \cref{align: erf equation}. Hence $H_\mathcal{V}$ is also ERF. 

     To conclude, let us consider any binary sequence $\mathbf b:=(b_i)_{i\geq 1}$, and define $U_{x,\mathbf b}$ to be the subset of $U_x:=\{u_i\}_{i\ge 1}$ consisting of the $u_i$ such that $b_i=1$ as in the proof of \cref{theorem: main result for micro supported actions}. We define $V_{x,\mathbf{b}}:=\{xv\mid v\in U_{x,\mathbf{b}}\}$, the disjoint union $$\mathcal{V}_\mathbf{b}:=V_{x,\mathbf{b}}\sqcup V_{x,\mathbf{b}}^a\sqcup\dotsb \sqcup V_{x,\mathbf{b}}^{a^{p-1}}$$ 
     and the subgroup
     $$H_\mathcal{V_\mathbf{b}}:= \langle ah_v\mid v\in V_{x,\mathbf{b}}\rangle.$$
     The same argument as in the proof of \cref{theorem: main result for micro supported actions} shows that $H_{\mathcal{V}_\mathbf{b}}=H_{\mathcal{V}_{\widetilde{\mathbf b}}}$ if and only if $\mathbf{b}=\widetilde{\mathbf b}$, which yields a continuum of pairwise distinct ERF subgroups of $G$ which are non-closed in the congruence topology.
\end{proof}

We note that if $G$ is a weakly branch group satisfying the assumption in \cref{theorem: main result for micro supported actions} and whose rigid vertex stabilizers all have torsion elements, then $G$ satisfies condition (i) in \cref{Theorem: uncountably many subgroups}.

\begin{example}[The first Grigorchuk group]
    The \emph{first Grigorchuk group} $\mathcal{G}\le \mathrm{Aut}~T_2$, is a group acting faithfully on the binary rooted tree introduced by Grigorchuk in \cite{GrigorchukExample}. It is generated by  $a,b,c,d\in \mathrm{Aut}~T_2$, where $a$ is the finitary automorphism acting as $\sigma:=(1\, 2)\in \mathrm{Sym}(2)$ on the first level of $T_2$ and with trivial sections everywhere else, and $b,c,d$ are defined recursively as
$$b=(a,c),\quad c=(a,d)\quad\text{and}\quad d=(1,b).$$
It is known to be an infinite branch torsion $2$-group \cite{GrigorchukExample}, 
and to have both the congruence subgroup property \cite{BG} and the subgroup induction property \cite{GW} (so in particular $\mathcal{G}$ is LERF). Thus, the first Grigorchuk group $\mathcal{G}$ satisfies the assumptions in \cref{theorem: main result for micro supported actions}, so it contains a continuum of pairwise distinct ERF subgroups which are non-closed in the profinite topology (by the congruence subgroup property).
\end{example}

\begin{example}[GGS-groups]
    The \emph{Grigorchuk-Gupta-Sidki groups}, \emph{GGS-groups} for short, are groups acting faithfully on the $p$-adic tree for an odd prime $p$ defined as a common generalization of the second Grigorchuk group \cite{GrigorchukExample} and the Gupta-Sidki $p$-groups \cite{GuptaSidki}. Let $e:=(e_1,\dotsc,e_{p-1})\in \mathbb{F}_p^{p-1}$. The GGS-group $G_e$ corresponding to the \emph{defining vector} $e$ is the group $G_e\le \mathrm{Aut}~T$ generated by $a$ and $b$, where  $a$ is again the finitary automorphism acting as $\sigma:=(1\,\dotsb\, p)\in \mathrm{Sym}(p)$ on the first level and with trivial sections everywhere else, and $b$ is the directed automorphism defined recursively as
$$b=(a^{e_1},\dotsc,a^{e_{p-1}},b).$$

It is known that a GGS-group is torsion if and only if $e_1+\dotsb+e_{p-1}\equiv 0$ mod$~p$ \cite[Theorem 1]{torsionGGS}, and it has the subgroup induction property if and only if it is torsion \cite[Theorems A and D]{FL}. Furthermore, every GGS-group with a non-constant defining vector is branch \cite[Lemma 3.3]{GGSGustavo}. Actually, for a GGS-group it is equivalent to be branch, to have the congruence subgroup property, and to have a non-constant defining vector; see \cite[Theorems~A and~C]{FGU} and \cite[Lemma 3.3]{GGSGustavo}. Every GGS-group with a non-constant defining vector satisfies  the assumption in \cref{Theorem: uncountably many subgroups} and it has the congruence subgroup property, so it also contains a continuum of pairwise distinct ERF subgroups which are non-closed in the profinite topology.
\end{example}

\begin{example}[Multi-GGS groups]
    The family of GGS-groups can be generalized to the family of the so-called \emph{multi-GGS groups}. A multi-GGS group $G$ is determined by~$r$ defining vectors $e_1,\dotsc,e_r\in \mathbb{F}_p^{p-1}$, where $e_{i}:=(e_{i,1},\dotsc,e_{i,p-1})$ for $1\le i\le r$.  Then $G$ is generated by the finitary automorphism $a$ defined as before and $b_1,\dotsc,b_r$ directed automorphisms given recursively by
$$b_1=(a^{e_{1,1}},\dotsc,a^{e_{1,p-1}},b_1),~\dotsc~,b_r=(a^{e_{r,1}},\dotsc,a^{e_{r,p-1}},b_r).$$

It was proved by Garrido and Uria-Albizuri that every non-constant multi-GGS group has the congruence subgroup property \cite[Theorem A]{GU}. Furthermore, if $r\ge 2$ the multi-GGS group is  branch. Moreover, if for some $1\le i\le p$ we have $e_{i,1}+\dotsb e_{i,p-1}\equiv 0$ mod$~p$ then the commutator subgroup $G'$ contains torsion and thus $G$ satisfies the assumption in \cref{Theorem: uncountably many subgroups}. If there exists another $1\le j\le p$ such that $e_{j,1}+\dotsb e_{j,p-1}\not\equiv 0$ mod$~p$ then the multi GGS-group is not torsion and hence, it does not have the subgroup induction property by \cite[Theorem~A]{FL}. Therefore, it is still open whether such a multi-GGS group is LERF or not, but it is not ERF by \cref{Theorem: weakly branch are not erf}.
\end{example}



\bibliographystyle{unsrt}

\end{document}